\documentclass[letterpaper,11pt,reqno,usegeometry, normalheadings]{scrartcl}
\usepackage{scrlayer-scrpage}
\usepackage[portrait,margin=1.25in]{geometry}
\usepackage[page]{totalcount}
\usepackage{scrlayer-scrpage}
\usepackage{etoolbox}
\patchcmd{\thebibliography}
  {\settowidth}
  {\setlength{\itemsep}{0pt plus 0.1pt}\settowidth}
  {}{}
\apptocmd{\thebibliography}
  {\small}
  {}{}
\title{\vspace{-.75in}
Stretch laminations and hyperbolic Dehn surgery}

\date{}
 
\usepackage[T1]{fontenc}
\usepackage{comment}
\usepackage{microtype}
\usepackage{pinlabel}
\usepackage{float}
\usepackage{graphicx}
\usepackage{subfig}
\usepackage{bbm}
\usepackage{amssymb, amsthm,thmtools,cite,mathtools,setspace,enumitem,tikz-cd}
\usepackage[margin=15pt,font=small,labelfont={bf,sf}]{caption}
\usepackage{hyperref}
\usepackage{cleveref}
\usepackage{thm-restate}
\usepackage{bm}
\usepackage[T1]{fontenc} 

\usepackage[p,osf, space=1.1]{erewhon}
\usepackage[varqu,varl]{inconsolata} 
\usepackage[utopia,vvarbb]{newtxmath}
\usepackage{AlegreyaSans}
\makeatletter
\newcommand\RedeclareMathOperator{%
	\@ifstar{\def\rmo@s{m}\rmo@redeclare}{\def\rmo@s{o}\rmo@redeclare}%
}
\newcommand\rmo@redeclare[2]{%
	\begingroup \escapechar\m@ne\xdef\@gtempa{{\string#1}}\endgroup
	\expandafter\@ifundefined\@gtempa
	{\@latex@error{\noexpand#1undefined}\@ehc}%
	\relax
	\expandafter\rmo@declmathop\rmo@s{#1}{#2}}
\newcommand\rmo@declmathop[3]{%
	\DeclareRobustCommand{#2}{\qopname\newmcodes@#1{#3}}%
}

\@onlypreamble\RedeclareMathOperator
\makeatother

\DeclareSymbolFont{bbold}{U}{bbold}{m}{n}

\newtheorem{thm}{Theorem}
\numberwithin{thm}{section}
\newtheorem{q}{Question}
\newtheorem{answer}{Answer}

\newtheorem{lem}[thm]{Lemma}

\newtheorem{prop}[thm]{Proposition}

\declaretheorem[style=remark]{remark}

\RedeclareMathOperator{\fill}{\mathtt{fill}}

\DeclareMathOperator{\inj}{inj}
\DeclareMathOperator{\diam}{diam}
\DeclareMathOperator{\radius}{radius}
\DeclareMathOperator{\sys}{sys}

\DeclareMathOperator{\area}{area}
\DeclareMathOperator{\len}{{len}}

\newcommand{\Z}{{\mathbb Z}}

\newcommand{\R}{{\mathbb R}}

\newcommand{\cC}{\mathcal{C}}

\newcommand{\cG}{\mathcal{G}}

\newcommand{\cQ}{\mathcal{Q}}
\newcommand{\cT}{\mathcal{T}}

\newcommand{\e}{\varepsilon}

\newcommand{\g}{\gamma}

\renewcommand{\a}{\alpha}

\newcommand{\fp}{\mathfrak p}

\newcommand{\interior}{{\mathsf{int}}}

\newcommand{\norm}[1]{\left\lVert#1\right\rVert}

\newcommand{\into}{\hookrightarrow}

\renewcommand{\d}{ \partial}

\renewcommand{\stretch}{{\mathsf{stretch}}}
\newcommand{\bsigma}{\overline{\sigma}}
\newcommand{\nlen}{\widehat{\len}}
\newcommand{\core}{\mathsf{core}}

\newcommand{\bs}{{\bm{s}}}

\author{Cameron Gates Rudd\footnote{Max-Planck-Institut für Mathematik in Bonn, email:\href{mailto:rudd@mpim-bonn.mpg.de}{rudd@mpim-bonn.mpg.de}}}
\begin{document}

\maketitle

\begin{abstract}
     We study maximal stretch laminations associated to certain best Lipschitz circle valued maps in Dehn surgery families of hyperbolic 3-manifolds. For these maps, we give a criterion based on the Thurston norm and Dehn filling slope length to determine when such a stretch lamination is a union of Dehn filling core curves. We use this to show there exist infinitely many examples where the homotopy class of the circle valued map includes a fibration and where the laminations have only closed leaves. This gives information about non-maximal horospherical orbit closures in the infinite cyclic covers associated to these fibrations.
\end{abstract}

\section{Introduction}

A map $f:M\to\R/\Z$ from a closed Riemannian manifold to the unit circle is best Lipschitz if it minimizes the Lipschitz constant among all homotopic maps. Associated to these maps is the set of points at which the Lipschitz constant is realized, called the stretch set.
For hyperbolic manifolds, the stretch set of any nontrivial best Lipschitz map contains a nonempty geodesic lamination $\Lambda_0$ with 1-dimensional leaves that depends only on the homotopy class. Homotopy classes of circle valued maps exactly correspond to cohomology classes $\rho \in H^1(M)$, so this associates to each nontrivial cohomology class a geometric object. We denote by $\Lambda(\rho)\subset\Lambda_0$ the maximal chain recurrent sublamination of $\Lambda_0$ and call it the cohomology stretch lamination of $\rho$; this is the main object of study in this paper. 

The topological, analytic, and dynamical properties of best Lipschitz maps and their stretch laminations have recently been studied by Gueritaud-Kassel \cite{GueritaudKassel}, Fathi-Siconolfi \cite{FathiSiconolfi}, Daskalopoulos-Uhlenbeck \cite{DaskalopoulosUhlenbeck}, and Farre-Landesberg-Minsky \cite{FarreLandesbergMinsky}. Gueritaud-Kassel's and Daskalopoulos-Uhlenbeck's work was motivated by Thurston's work on stretch maps between surfaces \cite{ThustonStretch}. Fathi-Siconolfi's work, in part, constructs $C^0$ 1-forms realizing Federer-Gromov's stable norm along a geodesic lamination.
Farre-Landesberg-Minsky relate these laminations to horospherical orbit closures in infinite cyclic covers. 

Previous work (see \cite{Eberlein}, \cite{Dalbo}, \cite{MaucourantSchapira}) identified non-maximal horospherical orbit closures with quasiminimizing points. These are points $v$ in the frame bundle whose orbits $\varphi^t(v)$ under the geodesic frame flow $\varphi^t$ escape to infinity at the fastest rate possible with only an additive error: $$d(\varphi^t(v),v)\geq t-B~\text{ for all } t\geq0$$ for some constant $B$. Such a trajectory $\{\varphi^t(v)\}_{t\geq0}$ is called a quasiminimizing ray.

Farre-Landesberg-Minsky show that in a hyperbolic infinite cyclic cover $\widetilde M_\rho\to M$ associated to a cohomology class $\rho\in H^1(M)$, the set of all accumulation points in the frame bundle of the base manifold $M$ of projected quasiminimizing rays from the frame bundle of the cover $\widetilde M_\rho$ projects to a sublamination of the cohomology stretch lamination $\Lambda(\rho)\subset M$. The set of lifts of these accumulation points in the frame bundle of $M$ to the frame bundle of the cover $\widetilde M_\rho$ is the $\omega$-limit set mod $\Z$ of the set of quasiminimizing points; this is denoted  $\cQ_\omega$. 

In the setting of infinite cyclic covers of closed hyperbolic surfaces, Farre-Landesberg-Minsky show the projection of $\cQ_\omega$ to the compact base is exactly the corresponding cohomology stretch lamination.
Furthermore, in the 2-dimensional setting, Farre-Landesberg-Minsky give a complete description of which measured laminations appear as cohomology stretch laminations, and thus also as projections of $\cQ_\omega$; namely, all that could.

One dimension higher however, nothing is known, though understanding the set of quasiminimizing points in geometrically infinite hyperbolic 3-manifolds has been of interest. In fact, the published literature contains a mistaken characterization of this set; see Remark \ref{remark:error} below. The goal of this paper is to rectify this lack of 3-dimensional examples by identifying cohomology stretch laminations in some Dehn surgery families.

In particular, we are interested in the following problem.

\begin{q}Fix a cusped hyperbolic 3-manifold $W$ and let $\bs_i$ be a sequence of Dehn filling slopes such that $M_i = W(\bs_i)$ is closed. Suppose that these slopes are compatible with cohomology classes $\rho_i^W\in H^1(W)$ in that they extend to the filled manifold $M_i$; we will call such classes surgery classes and the slopes $\bs_i$ compatible. What can we say about cohomology stretch laminations of the extensions $\rho_i$ of the surgery classes $\rho_i^W$?
\end{q}

Informally, this paper provides the following (partial) answer.

\begin{answer} For large Dehn fillings and compatible cohomology classes ``with nonempty boundary,'' as long as the Thurston norm is not much larger than the length of the filling slope, the cohomology stretch laminations of extended classes are unions of Dehn filling cores. In particular, they are geodesic links.
\end{answer}

This leads to geometrically infinite 3-dimensional examples where the $\omega$-limit set mod $\Z$ of quasiminimizing points projected to the compact base manifold can be exactly identified; see Theorem \ref{thm:fiberedexamples} and the subsequent remark. Topological information about these laminations then has consequences for the behavior of horospherical orbits in the corresponding infinite cyclic covers. For instance, because every leaf is uniformly isolated, Corollary 7.23 in \cite{FarreLandesbergMinsky} says that for examples in which the stretch laminations are geodesic links, such as those in Theorem \ref{thm:fiberedexamples}, the projection of the horospherical orbit of a frame $v$ in $\cQ_\omega$ for the corresponding cover $\widetilde M_\rho$ does not accumulate at the projection of $v$ in $\widetilde M_{\rho}$.

\begin{remark}\label{remark:error}
    Fibered examples in which the cohomology stretch lamination is a geodesic link, again as in Theorem \ref{thm:fiberedexamples}, are counterexamples to Theorem\nopagebreak~1.3  in \cite{LecuireMj}. Section 3 in \cite{LecuireMj} explains how their Theorem 1.3 relates to the set of quasiminimizing points. Combining that discussion in the fibered case with the results of \cite{FarreLandesbergMinsky}, the $\omega$-limit set mod $\Z$ of quasiminimizing points, which is a sublamination of $\Lambda(\rho_i)$, contains the geodesic tightening of the suspension flow associated to the pseudoAnosov monodromy. The proof was previously discovered by Farre and Minsky to be incomplete; see the Erratum in the arXiv version of \cite{LecuireMj}.
    \end{remark}

\subsection{Results}
    Our analysis relies on understanding how stretch laminations interact with certain thick-thin decompositions.
    In Section \ref{sec:stretchtubes}, we examine how stretch laminations behave near large tubes. This allows us to decompose the problem of finding the stretch lamination into distinct thick and thin parts in large Dehn fillings. 
    For this, understanding the complexity of the extended cohomology classes $\rho_\bs$ for a Dehn filled manifold $M= W(\bs)$ in the thick and thin parts separately is key.

    To do this, it is useful to have comparisons between geometric and topological norms on the cohomology of hyperbolic 3-manifolds. Results of this sort have previously been obtained by Brock-Dunfield in \cite{BD} and Han in \cite{Hans}. The norm comparisons proved in Section \ref{sec:thickstablethurstoncomparison}, Theorem \ref{thm:chainnormineqs} in particular, are similar to theirs. \footnote{We need an analogue of their lower bound on the Thurston norm, but for relative classes in the thick part of a cusped manifold. It's worth noting their constants are explicit and depend only on the systole; that is not the case here. Our comparison depends on the cusped manifold and a choice of Margulis constant.}

    One can measure the geometric complexity of an extended class $\rho_\bs\in H^1(M)$ in the Dehn filling $M=W(\bs)$ using the (dual) stable norm, also called the comass, $$\norm{\rho_\bs}_0:=\sup\limits_{\g\in \pi_1 M}\frac{\rho_\bs(\g)}{|\g|_{M}},$$ where $|\g|_{M}$ is the minimal length of a loop in the free homotopy class defined by $\g$. For a given Margulis constant $\mu$, we can also define thick stable norms $$\norm{\rho_\bs}_\mu = \sup\limits_{\g\in\pi_1M}\frac{\rho_\bs(\g)}{|\g|_{M_{\geq\mu}}},$$ where the length $|\g|_{M_{\geq\mu}}$ is defined as the minimal length of a loop in the free homotopy class of loops representing $\g$ with the restriction that they must stay in the $\mu$-thick part of $M$.

    Associated to a filling slope $\bs$ is a certain normalized length determined by the Euclidean geometry of the cusps; see Section \ref{sec:dehnfill} for the definition.
    We relate these complexity measures with cohomology stretch laminations in the following theorem.

\begin{restatable}{thm}{laminationsintubes}\label{thm:laminationsintubes}
    Let $W$ be a cusped hyperbolic 3-manifold. There is a Margulis constant $\mu$ satisfying $\mu<\min\{\sys(W)/4,\log(3)/1.2\}$ and a constant $L>0$ both depending only on $W$ such that if $\rho$ is a nontrivial surgery class compatible with a complete slope $\bs$ with total normalized length $\nlen(\bs)>L$ and $\rho_\bs$ is the extended class, then either $\Lambda(\rho_\bs)$ is a union of Dehn filling core curves, or $||\rho_\bs||_0\leq 3||\rho_\bs||_\mu$.
\end{restatable}

See Theorem \ref{thm:onlycores} for a version of this result for closed manifolds without reference to Dehn fillings. The basic idea in proving Theorem \ref{thm:onlycores} is that (under the hypotheses of the theorem) if there is a leaf of $\Lambda(\rho_\bs)$ intersecting the thick part, then either the stable norm is realized in the thick part or there is a leaf of $\Lambda(\rho_\bs)$ that travels from the thick part deep into a Margulis tube. This time spent traveling into the tube enables a comparison between the thick and regular stable norms. For Dehn fillings with sufficiently long slope, the hypotheses of Theorem \ref{thm:onlycores} hold, leading to Theorem \ref{thm:laminationsintubes}.

    For a sufficiently small Margulis constant $\mu$, the $\mu$-thick stable norm of the extension $\rho_\bs\in H^1(W(\bs))$ of a surgery class $\rho\in H^1(W)$ for $\bs$ a large Dehn filling slope is controlled by the Thurston norm of the class $\rho\in H^1(W)$.

    \begin{restatable}{thm}{thickstableThurston}\label{thm:thickstableThurston} Let $W$ be a cusped hyperbolic 3-manifold and $\mu<\min\{\sys(W)/4,\log(3)/1.2\}$ a Margulis constant. There are constants $C,L>0$ depending only on $W$ and $\mu$ such that if $\rho\in H^1(W)$ is a surgery class with compatible slope $\bs$ satisfying $\nlen(\bs)>L$ and $\rho_\bs$ is the extended class, then  $||\rho_\bs||_\mu\leq C||\rho||_{Th}.$
    \end{restatable}

    This estimate allows one to make use of Theorem \ref{thm:laminationsintubes} with only \emph{topological} information about the complexity of the compatible cohomology classes in $H^1(W)$.
    Then, Dehn filling length estimates of Futer-Purcell-Schleimer, refining an asymptotic result of Neumann and Zagier, can be used to obtain bounds on the stable norm in the thin part. Together, these give tools to certify that the cohomology stretch lamination is a finite union of closed curves. For certain slopes, this leads to Theorem \ref{thm:littleoh} below.

    \begin{figure}[h]
        \centering
        \includegraphics[scale=.3]{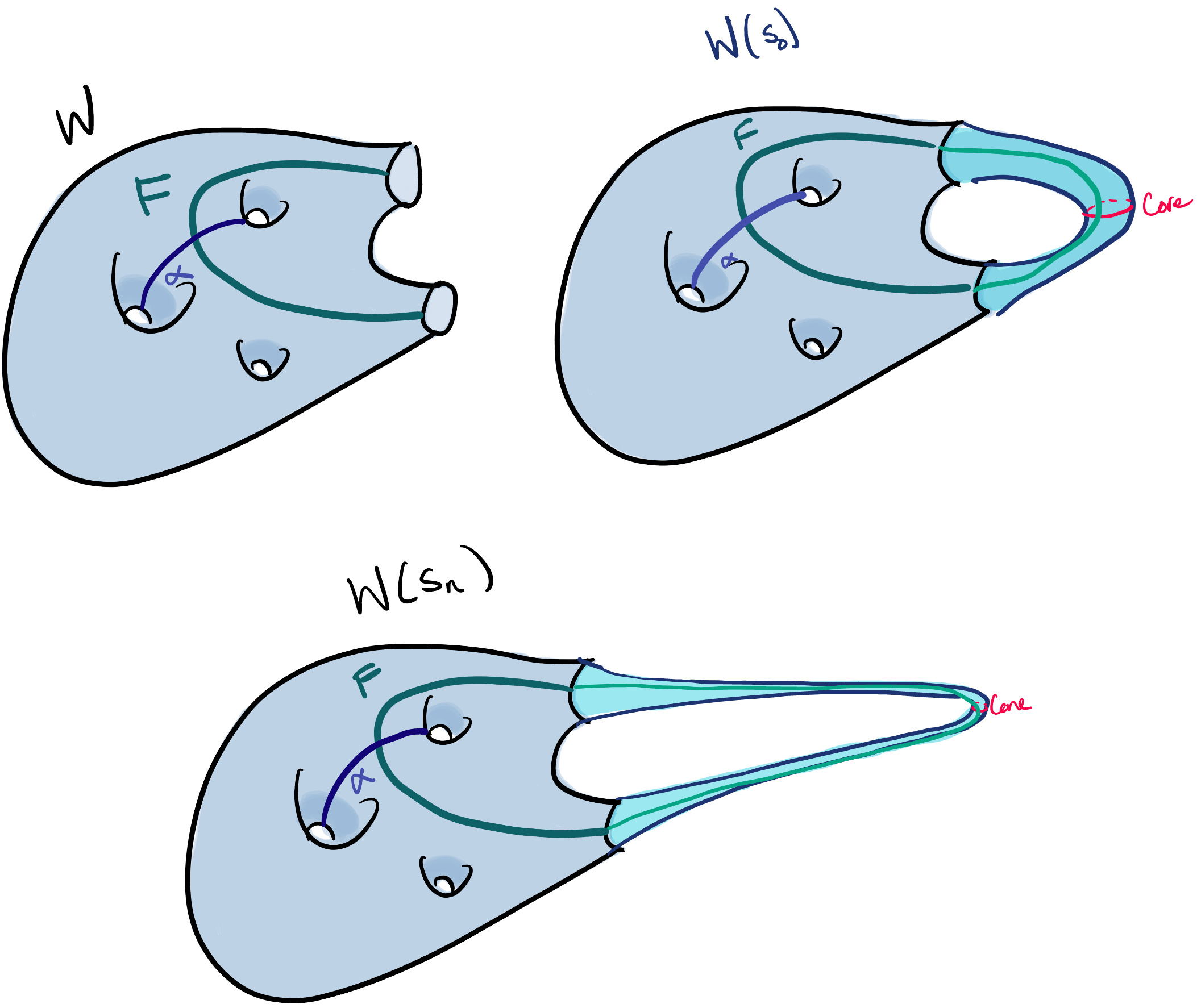}\caption{A 2-dimensional cartoon illustrating how the Lipschitz constant blows up during the surgery while the Lipschitz constant in the thick part stays bounded. The relative cycle $F$ in the initial manifold is Poincaré dual to a cohomology class $\rho$. It extends to a cycle after a cylinder is attached. Pinching the core curve $c$ of the cylinder causes the Lipschitz constant to blow up.  One should compare the ratio $\rho(\a)/\len(\a)$ for the curve in the thick part intersecting $F$ once to the ratio $\rho(c)/\len(c)$ for the core of the tube. Because for any representative form $\omega$, we have $||\omega||_\infty \geq 1/\len(\g)\int_\g\omega$, this forces the Lipschitz constant to be large along $c$.}\label{fig:pinchcurves}
    \end{figure}

    \begin{figure}[h]
        \centering
        \includegraphics[scale=.3]{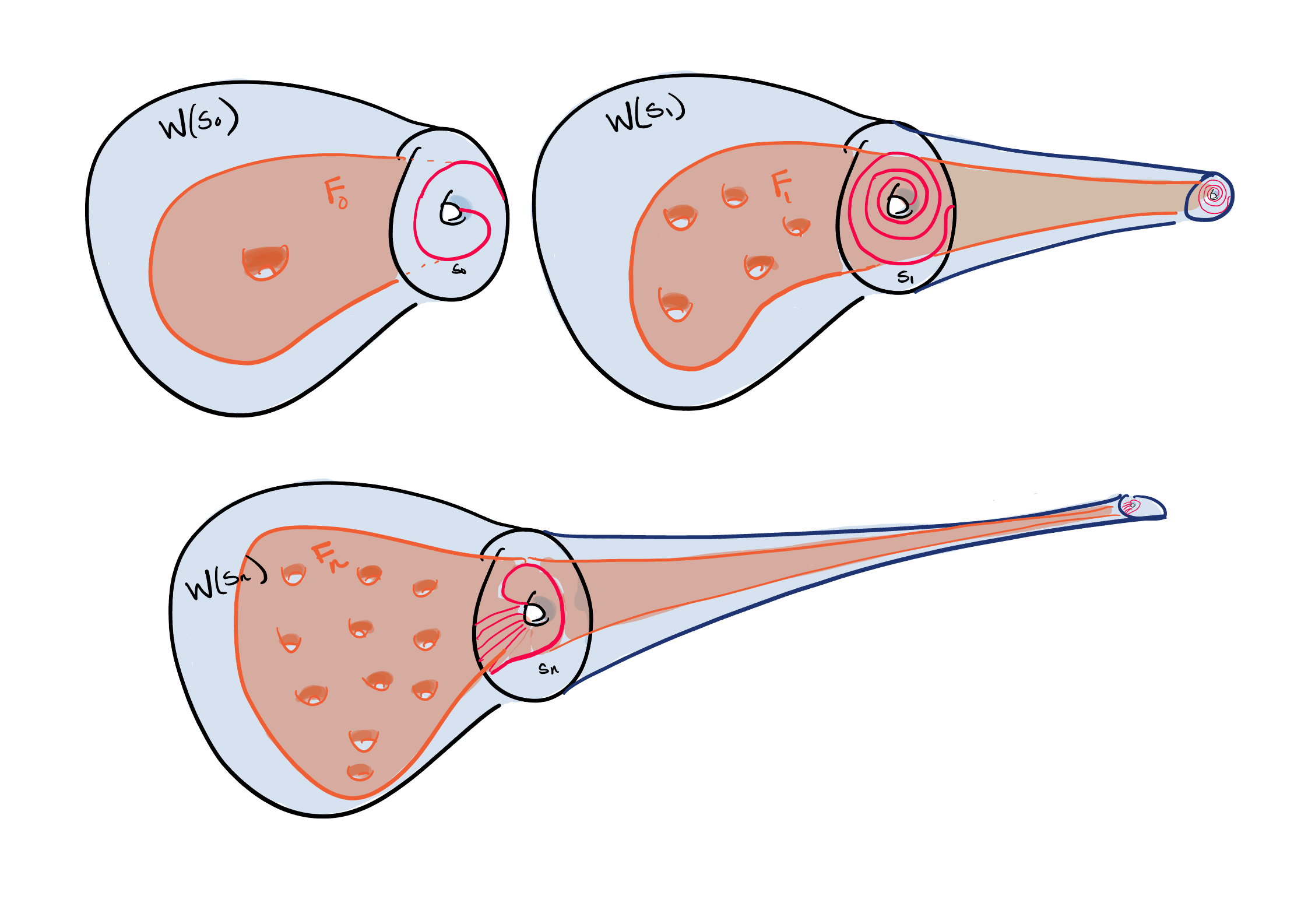}\caption{The idea in 3-dimensions is to look at a sequence of cohomology classes defining Dehn fillings. In the Poincaré dual picture, the surface dual to the class is a surface with boundary; the boundary defines the filling slope. This mimics the 2-dimensional cartoon, except now the cohomology class changes as one "pinches" the core curve. The key is that the Lipschitz constant in the thick part, whose geometry is stable, can be controlled by the Thurston norm. For sequences of classes where the growth of the Thurston norm is slower than the growth of the Lipschitz constant in the thin part, one can (eventually) show that the thin part determines the stretch lamination.}\label{fig:dehnfillingfamily}
    \end{figure}

    For simplicity, we now restrict to a convenient type of surgery class.
    A class $\rho\in H^1(W)$ is a $0$-surgery class if it is Poincaré dual to a surface that intersects each boundary component of $W$ exactly once in an essential closed curve. The boundary curves of this dual surface define a Dehn filling.
    Applying Theorem \ref{thm:laminationsintubes} in $0$-surgery families, we can use the norm comparison from Theorem \ref{thm:thickstableThurston} to prove the following:

    \begin{restatable}{thm}{littleoh}\label{thm:littleoh} Let $W$ be a cusped hyperbolic 3-manifold. Let $\rho_i^W$ be a sequence of ~$0$-surgery classes in $H^1(W)$  with boundary slopes $\bs_i$ whose total normalized lengths are going to infinity.  Let $\rho_i$ be the extension of $\rho_i^W$ to the Dehn filled manifold $M_i = W(\bs_i)$.  If the Thurston norms of the classes $\rho_i^W$ grow subquadratically in the total normalized length of the boundary slopes: $$||\rho_i^W||_{Th} = o\left(\nlen(\bs_i)^2\right),$$ then there is an $N>0$ such that for all $i>N$, the cohomology stretch lamination $\Lambda(\rho_i)\subset M_i$ is a union of Dehn filling core curves.
    \end{restatable}

    Note that we have replaced a geometric condition about the class in the filled manifold with a condition on the Thurston norm and slope length in the cusped manifold, which are much more accessible. A specific case where Theorem \ref{thm:littleoh} applies is given in Theorem \ref{thm:fiberedexamples} below.

    The idea is illustrated in Figures \ref{fig:pinchcurves} and \ref{fig:dehnfillingfamily}. Essentially, for these $0$-surgery Dehn fillings, the filling core curves force the stable norm to grow quadratically in the filling slope length. Using our norm comparison, the thick stable norm is controlled by the Thurston norm, which we assume is growing slower than the squared slope lengths. Thus, eventually the ordinary stable norm is much larger than the thick stable norm, so Theorem \ref{thm:laminationsintubes} identifies the stretch locus as a union of Dehn filling curves.

    This condition on the Thurston norm is not too strong, and indeed this theorem applies to a family of hyperbolic 3-manifolds constructed by Brock and Dunfield in \cite{BD}. This then gives infinitely many closed \emph{fibered} hyperbolic 3-manifolds $M_n$ along with fibered cohomology classes $\rho_n\in H^1(M_n)$ for which the cohomology stretch laminations $\Lambda(\rho_n)$ are a finite union of simple closed geodesics. 

    \begin{restatable}{thm}{fiberedexamples}\label{thm:fiberedexamples}
        There exists a 2-cusped fibered hyperbolic 3-manifold $W$ and infinitely many $0$-surgery classes $\rho_n^W\in H^1(W)$ satisfying the conditions of Theorem \ref{thm:littleoh}. The $0$-surgery classes can all be taken to be fibrations so that the filled manifolds also fiber and such that the resulting cohomology stretch laminations are the union of the two corresponding Dehn filling core curves. 
    \end{restatable}

    By Proposition 3.5 and Remark 3.6 in \cite{FarreLandesbergMinsky}, this theorem completely identifies the projection to the compact base manifold of the $\omega$-limit set mod $\Z$ of quasiminimizing points in the infinite cyclic covers associated to these fibrations. As noted in Remark \ref{remark:error}, these examples contradict a previous result that implied these laminations were much larger.

    \subsection*{Notation}
    Throughout, the constant $C$ is positive and depends on a fixed cusped manifold $W$ and a choice of Margulis constant $\mu$. We denote by $\len(\g)$ the length of a curve $\g$ and for $\g$ a loop or multi-loop in $M$ we denote by $|\g|_M$ the length of the shortest loop or multi-loop in the free homotopy class of $\g$ in $M$.

    \subsection*{Overview of paper} In Section \ref{sec:dehnfill}, we discuss some geometric aspects of Dehn filling, namely Brock-Bromberg's Drilling Theorem, Neumann-Zagier's length estimates, and effective refinements of both results due to Futer-Purcell-Schleimer. We also introduce "deep" Margulis constants, which we require in sections \ref{sec:stretchtubes} and \ref{sec:mainthms}. In Section \ref{sec:zerosurgery}, we discuss cohomology classes and compatible Dehn fillings. In Section \ref{sec:laminations}, we outline basic definitions and results on laminations and best Lipschitz maps. In Section \ref{sec:stablenorm}, we discuss the stable norm and its connection to best Lipschitz mappings. In Section \ref{sec:thickstablethurstoncomparison}, we relate the thick stable norm to the Thurston norm; the main result here is Theorem \ref{thm:chainnormineqs}. In Section \ref{sec:bigtubes}, we prove some preliminary results about curves near deep Margulis tubes. In Section \ref{sec:stretchtubes}, we study how chain recurrent stretch laminations behave near deep tubes. This includes the key estimate, in Theorem \ref{thm:onlycores}, relating the stable norm to the question of identifying strech laminations. In Section \ref{sec:mainthms}, we combine these results to obtain the main theorems on Dehn fillings discussed in the introduction.

    \section*{Acknowledgements} The author thanks James Farre for suggesting the problem of identifying stretch laminations in some 3-dimensional examples and for helpful discussions. The author thanks Nathan Dunfield and James Farre for comments on an earlier draft. The author thanks Xiaolong Hans Han for many helpful conversations on related topics and thanks Yair Minsky for an interesting conversation on this topic.  This work began at the conference \emph{Groups, Geometry, and Dynamics-- Celebrating the mathematics of Ursula Hamenstädt} at Congressi Stefano Franscini and the author thanks the organizers of that event. The author is grateful to the Max-Planck-Institute für Mathematik in
    Bonn for its hospitality and financial support.

    \section{Cusps, tubes, and fillings}\label{sec:dehnfill}

    Let $W$ be a cusped hyperbolic 3-manifold. Let $\{\cC_i\}$ be the cusps of $W$. Let $\{T_i\} = \{\d\cC_i\}$ be the cusp horotorus cross-sections of maximal area. A homotopy class of simple closed curve on the torus $T_i$, or equivalently a primitive homology class $s_i\in H_1(T_i;\Z)$, is called a slope. The length $|s_i|_{T_i}$ of a slope $s_i$ on $T_i$ is the minimal Euclidean length on the flat torus $T_i$ of a curve homotopic to $s_i$. The cusp length $\len_{\cC_i}$ is defined by scaling the length by area: $$\len_{\cC_i}(s_i)=\frac{|s_i|_{T_i}}{\sqrt{\area(T_i)}};$$ this quantity is unchanged if one chooses a different horotorus cross-section.
    A \textbf{complete slope} $\bs = (s_i)$ is a vector of slopes, one for each cusp.
    The \textbf{total normalized length} $\nlen(\bs)$ of a complete slope $\bs$ is defined by $$\nlen(\bs)^{-2}:=\sum_i\len_{\cC_i}(s_i)^{-2}.$$

    Let $\d W = \cup_i T_i$ and let $\bs = (s_i)$ be a complete slope. By attaching a solid torus to each boundary component of $W$ by a map identifying a meridian of the solid torus with the curve on the boundary torus specified by the slope, one obtains from $W$ a closed 3-manifold $M = W(\bs)$ that for long slopes is hyperbolic, see \cite{Thurston} or \cite{HodgsonKerckhoff}. The core curves of the Dehn filling solid tori form a link in the filled manifold.

    In the Dehn filled manifold $M = W(\bs)$ with its hyperbolic structure, around the link of Dehn filling core geodesics is a \textbf{collection of maximal tubes.} A collection of tubes $\cT_i$ is maximal if the interiors of the tubes are disjoint and no tube can be enlarged without its interior bumping into another tube (or possibly itself). See \cite{FuterPurcellSchleimer} Definition 4.2 for a recipe to construct maximal collections of embedded tubes around geodesic links. We assume throughout that a maximal collection of tubes is constructed  around a link in this way.
    Given a union of tubes $\cup_i\cT_i$ in $M$, we write $\core(\cup_i\cT_i)$ for the link of core geodesics in $M$. A \textbf{collection of maximal cusps} are defined similarly to collections of maximal tubes.

    The reciprocal squared total normalized length of a complete slope closely approximates the total length of the core geodesics of the Dehn filling solid tori in the filled manifold equipped with its hyperbolic structure. An asymptotic version of this was first proved in \cite{NeumannZagier} and has recently been refined by \cite{FuterPurcellSchleimer}.

    \begin{thm}[Quantiative Neumann-Zagier, Corollary 6.13 \cite{FuterPurcellSchleimer}]\label{thm:qNZ}    Let $W$ be a cusped finite volume hyperbolic 3-manifold. Suppose $M = W(\bs)$ is obtained from $W$ by Dehn filling with slope $\bs$ and the total normalized length $\ell = \nlen(\bs)$ of the slope $\bs$ is greater than 7.823. Let $c$ be the link in $M$ of core curves of the filling. Then $$\frac{2\pi}{\ell^2+16.17}<|c|_M < \frac{2\pi}{\ell^2-28.78}.$$
    \end{thm}

    Given a finite volume hyperbolic 3-manifold $M$, and some small number $\mu>0$ there is a pair of natural compact manifolds with boundary called the $\mu$-\textbf{thick part} (or just thick part) of $M$: $$M_{\geq \mu} = \{x\in M~:~ \inj_x\geq \mu/2\}$$ and the $\mu$-\textbf{thin part} (or just thin part)$$M_{\leq \mu} = \{x\in M~:~ \inj_x\geq \mu/2\}.$$  There is a universal constant $\mu_3>0$, called the 3-dimensional Margulis constant, such that if $\mu\leq \mu_3$ then the $\mu$-thick part of $M$ is compact with toroidal boundary and the $\mu$-thin part is a disjoint union of tubes and cusps. We call such a positive $\mu<\mu_3$ a \textbf{Margulis constant}.

It will be useful to have a Margulis constant that decomposes the tubes in a particular way. Let $M$ be closed. A Margulis constant $\mu$ is \textbf{$(D,t)$-deep} if the $\mu$-thin part $M_{\leq\mu}$ is a union of disjoint tubes such that the corresponding collection of maximal tubes $\cT_i$ have boundary $T_i=\d \cT_i$ distance at least $D$ from the torus $T^{\mu}_i= \d M_{\geq\mu}\cap \cT_i$ of injectivity radius exactly $\mu$, and such that $T_i^\mu $ is at least distance $t$ from the core of the tube. See Figure \ref{fig:deeptube} for a schematic.

\begin{figure}[h]
    \centering
    \includegraphics[scale=.35]{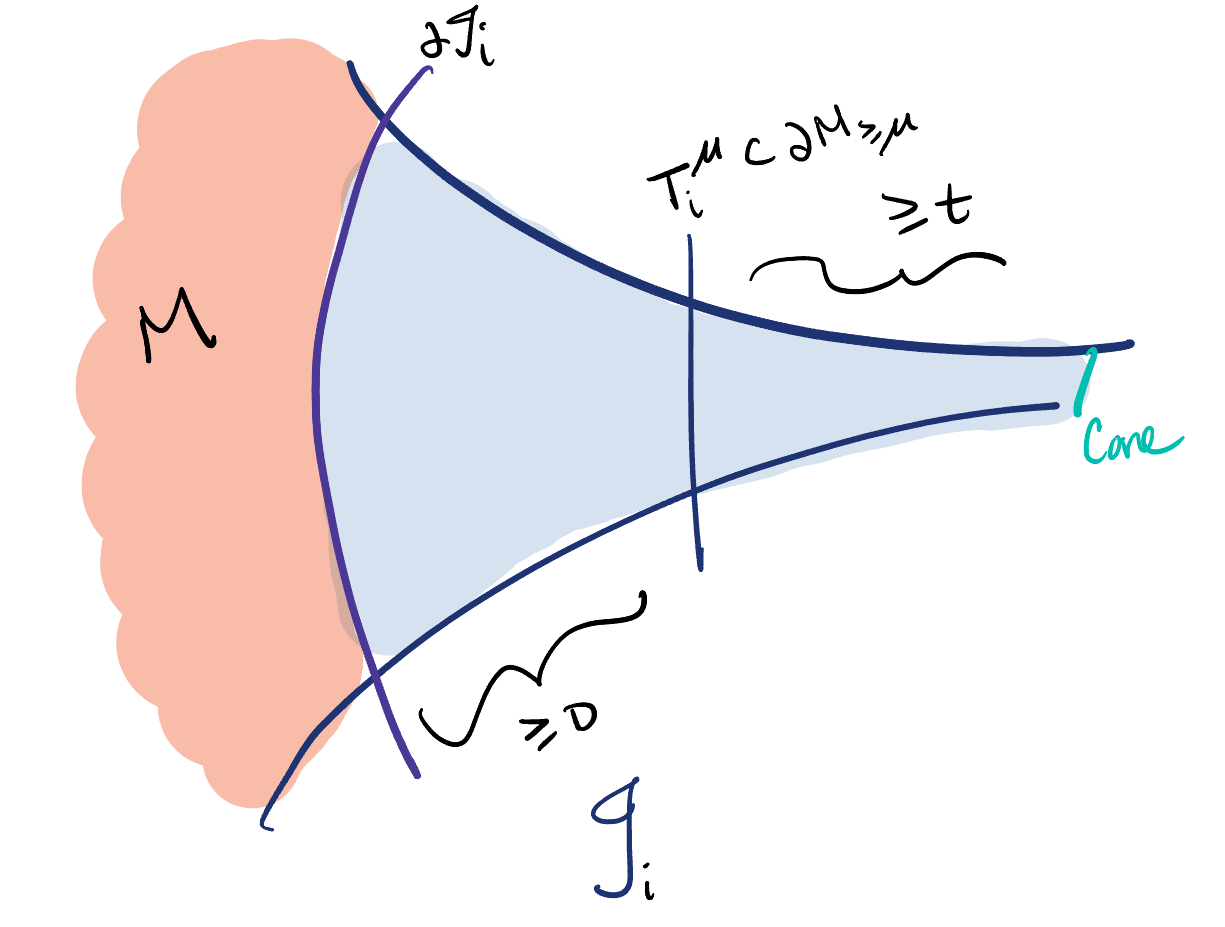}\caption{Schematic of a $(D,t)$-deep Margulis constant $\mu$.}\label{fig:deeptube}
\end{figure}

If $W$ is cusped, we say a Margulis constant $\mu$ is \textbf{$L$-persistantly $(D,t)$-deep} if for all complete slopes $\bs$ with total normalized length greater than $L$, the Margulis constant $\mu$ is $(D,t)$-deep for $W(\bs)$.

When a closed hyperbolic manifold $M$ is obtained from a cusped manifold $W$ by Dehn filling, there is a relationship between the geometry of the thick part of $M$ and the thick part of $W$. This is encoded in the Drilling Theorem of Brock and Bromberg \cite{BrockBromberg}. Recently, an effective version of this result has been obtained by Futer, Purcell, and Schleimer in \cite{FuterPurcellSchleimer}.

\begin{thm}[Effective Brock-Bromberg Drilling theorem, Theorem 1.2 \cite{FuterPurcellSchleimer}]\label{thm:drill} Fix $J>1$ and $\log3>\e>0$. Then there is an explicit $\ell_0 = \ell_0(J,\e)$ such that for every finite volume hyperbolic 3-manifold $M$ the following holds: Let $\Sigma$ be a geodesic link with total length less than $\ell_0$. Then the exterior $M-\Sigma$ admits a hyperbolic structure $W$ and there are natural $J$-biLipschitz diffeomorphisms $$M_{\geq\e}\into W_{\geq\e/1.2}\text{ and } W_{\geq\e}\into M_{\geq\e/1.2}.$$ When both maps are defined, their composition is the identity.
\end{thm}

As an application of the drilling theorem, we get that large Dehn fillings admit $L$-persistently $(D,t)$-deep Margulis constants for large $L = L(D,t)$ and any $D,t>0$.

\begin{prop}\label{prop:deepmargulis}  Let $W$ be a cusped hyperbolic 3-manifold. Fix constants $D>0, t>0$.
    Then there are constants $\mu$ and $L$ depending on $W$, $D$, and $t$ such that $\mu$ is $L$-persistantly $(D,t)$-deep. The Margulis constant can be taken to be arbitrarily small.
\end{prop}
\begin{proof}
    Fix a collection of maximal cusps for all the cusps in $W$.
    Let $2\mu<\mu'<\mu_3/2$ be a pair Margulis constant such that in $W$, for each cusp $\cC$, the distance from the boundary of $\cC_{\geq2\mu'}$ to the boundary of $\cC_{\leq2\mu}$ is greater than $2D+1$. One can choose $\mu$ to be arbitrarily small. 

    Take $J=2$ and $\e=2\mu$ in the Drilling Theorem \ref{thm:drill}. By quantitative Neumann-Zagier (Theorem \ref{thm:qNZ}), the short core curve condition can be replaced with a long filling slope condition. Thus there is an $L>0$ such that for filling slopes $\bs$ with $\nlen(\bs)>L$, there is a $2$-biLipschitz map $ W_{\geq2\mu} \into W(\bs)_{\geq2\mu/1.2}\subset W(\bs)_{\geq\mu}$.

    For each cusp $\cC$, the image of $\cC\cap W_{\geq\mu}\cap W_{\leq2\mu'}$ under this embedding has distance strictly greater than $D$ between boundary components and is contained in $W(\bs)_{\leq\mu_3}$ so is contained in the maximal tube $\cT$ in the filled manifold corresponding to the cusp $\cC$.

    Therefore, $\mu$ is a Margulis constant that for each maximal tube $\cT$ around a Dehn filling core curve, the thick part of the tube $\cT_{\geq\mu}$ has distance at least $D$ between endpoints, or if the core is contained in the thick part, then the radius of the tube is greater than $D$. The $\mu$-thin part of the tube can be made arbitrarily large as the normalized length of the filling slope grows by the tube radius estimate in Corollary 2 in \cite{Meyerhoff}, which applies due to the slope length and core curve length estimate in Theorem \ref{thm:qNZ}.
    Thus after increasing $L$, for all slopes with total normalized length larger than $L$, the Margulis constant $\mu$ is $(D,t)$-deep.
\end{proof}

We will also need a uniform bound on the diameters of the boundaries of tubes in hyperbolic Dehn fillings. This is contained in a result of Hodgson and Kerckhoff.

\begin{lem}\label{lem:diambound} Let $W$ be a cusped hyperbolic 3-manifold. Then there is a constant $D$ depending on $W$ such that for any complete slope $\bs$ such that $W(\bs)$ is hyperbolic, each maximal tube around a core curve of the Dehn filling has boundary with diameter less than $D$.

\end{lem}
\begin{proof}
    This is essentially contained in the proof of Theorem 3.5 in \cite{HodgsonKerckhoff}.
    By Theorem 4.26 in \cite{FuterPurcellSchleimer}, in Dehn fillings of $W$ there is a lower bound on the injectivity radius at points on the boundary of maximal tubes which increases as the tube radius grows and is for sufficent radius nonnegative. Because the tube radii grow as the core curves becomes shorter (Corollary 2 \cite{Meyerhoff}), there is an $L$ such that for all slopes $\bs$ with total normalized length greater than $L$, this injectivity radius lower bound is positive and holds for the boundary of maximal tubes in the Dehn filled manifolds $W(\bs)$. As this only excludes finitely many slopes, there is a uniform lower bound for all hyperbolic fillings.
    By the injectivity radius lower bound, if the diameters of the tori diverged, their areas would diverge as well.
    If the area of the maximal tube boundary tori diverged, again using the injectivity radius bound, one could find an embedded tubular neighborhood around each such torus in the sequence of Dehn filled manifolds such that the volumes of the tubular neighborhoods diverged. This would make the volume of the Dehn filled manifolds diverge, but this would contradict that the volumes of the Dehn fillings are all bounded above by the volume of $W$.
        \end{proof}

    \section{Surgery classes and 0-surgery slopes}\label{sec:zerosurgery}
    Let $\rho\in H^1(W)$ be a cohomology class. In addition to a homotopy class of maps to $\R/\Z$, the class $\rho$ defines a Poincaré-Lefschetz dual homology class $a \in H_2(W,\d W)$. A representative of $a$ is given by the preimage of a regular value of a $C^1$ map $W\to \R/\Z$ in the relevant homotopy class.

    We want to know which boundary slopes $\bs$ are compatible with $\rho$; that is, to which Dehn fillings the class extends. Let $a$ be the Poincaré-Lefschetz dual class to $\rho$ in $H_2(W,\d W)$ and $F$ a properly embedded incompressible surface representing $a$ with each component of $F$ homologically essential (relative the boundary) and each boundary component of $F$ an essential closed curve. Then, each boundary component of $F$ puts a constraint on possible compatible slopes $\bs$. A complete slope $\bs$ is \textbf{compatible} with $\rho$ if for each such surface $F$ and boundary component $T$ of $W$, $F\cap T$ is a union of curves parallel the curve specified by $\bs$. After Dehn filling with a compatible slope, the surface $F$ has boundary components that can capped off by meridian disks; the result is a closed surface $F_\bs$. The \textbf{extended class} $\rho_\bs$ is the Poincaré dual of this capped off surface $F_\bs$ in $W(\bs)$. In this case, we say the class $\rho$ is a \textbf{surgery class} and that it is \textbf{compatible} with the total slope $\bs$.

    We say the class $\rho$ is a \textbf{$0$-surgery class} defining a slope $\bs = \d a \in H_1(\d W)$ if there is a properly embedded incompressible surface $F$ representing the Poincaré dual class $a$ with boundary intersecting \emph{every} boundary component of $W$ exactly once in an essential  simple closed curve. The slope is then the vector of homology classes on the boundary of these boundary curves; we call this the \textbf{boundary slope} of the $0$-surgery class. As above, the class, $\rho$ extends to a class in $H^1(W(\bs))$ with the extension of $\rho$, denoted $\rho_\bs$, being obtained from the capped off surface, denoted $F_\bs$, by taking the Poincaré dual.
    Note that the geometric intersection of $F_\bs$ with each Dehn filling core curve is exactly 1.

\section{Geodesic laminations and best Lipschitz maps}\label{sec:laminations}

        A \textbf{geodesic lamination} with $1$-dimensional leaves in a Riemannian manifold $M$ is a closed subset of $M$ that is foliated by geodesics with a continuous local product structure in the form of charts $\varphi_i:U_i\to \R\times \R^2$ covering a neighborhood of the lamination such that in these charts the lamination is of the form $\R\times B$ for $B\subset\R^2$ and the transition functions $\varphi_{ij}$ are of the form $\varphi_{ij}(t,x,y)= (f_{ij}(t,x,y),g_{ij}(x,y), h_{ij}(x,y))$. We refer to these simply as geodesic laminations and always assume the leaves are 1-dimensional.
        A geodesic lamination $\Lambda$ is \textbf{recurrent} if for all $\e>0$ and any point $p$ contained in a leaf $\lambda$ of $\Lambda$, there is a unit speed loop $\g$ in $M$ through
        $p$ such that at every point $x\in\g$, the unit length parametrized segment of $\g$ around $x$ is $\e$-close to a unit length path in $\lambda$ in the $C^1$ sense.

        This essentially means one can return arbitrarily close to $p$ arbitrarily often by traveling (in one direction) along the leaf $\lambda$.

        A geodesic lamination $\Lambda$ is \textbf{chain recurrent} if for every $\e>0$ and every point $ p$ in $\Lambda$, there is a unit speed loop $\g$ in $M$ through $p$ such that for each $x$ in $\g$, the unit length parametrized segment of $\g$ around $x$ is $\e$-close to a unit length path in $\Lambda$ in the $C^1$ sense.

        Essentially, this means one can return arbitrarily close to $p$ arbitrarily often by traveling (in one direction, which may be only locally defined via a small oriented transversal) along leaves of $\Lambda$, but in contrast to the recurrent case one can occasionally jump between nearby leaves.

        Our main object of study are chain recurrent geodesic laminations associated to cohomology classes. 

        A map $f$ in a homotopy class of maps $M\to \R/\Z$ is called \textbf{best Lipschitz} if it has minimal possible Lipschitz constant in its homotopy class. Note that best Lipschitz maps are generally far from unique as the map can be deformed near points where the Lipschitz constant is small. Because the map $f$ may not be $C^1$, the Lipschitz constant is defined as follows: 

        For $E$ a subset of a Riemannian manifold $M$ and $u : M \to \R/\Z$, the Lipschitz constant on $E$ is defined to be
        $$L_u(E) := \inf\{L \in\R~:~ d_{\R/\Z}(u(x)-u(y)) \leq Ld_M(x, y) \text{ for all } x,y\in E\}.$$
        The Lipschitz constant of $u$ is $$L_u := L_u(M).$$
        We call the infimal Lipschitz constant over all maps $M\to \R/\Z$ in the homotopy class defined by a cohomology class $\rho\in H^1(M)$ the \textbf{cohomology stretch} of $\rho$ and write $L(\rho)$ for its value: $$L(\rho)=\inf \{L_u~:~ u \text{ in homotopy class defined by }\rho\}.$$

        The local Lipschitz constant of $u:M\to\R/\Z$ is a function $L_u:M\to\R$ given by $$L_u(x)=\lim\limits_{r\to0}L_u(B_r(x)).$$
        The map $L_u:M\to\R$ is upper semicontinuous (Lemma 2.9 \cite{GueritaudKassel}).

        Associated to any best Lipschitz map $f:M\to\R/\Z$ is a \textbf{maximal stretch set} $$\stretch(f) = \{x\in M~:~L_f(x)=L_f\}$$ of points at which the Lipschitz constant is realized. The intersection of maximal stretch sets over all best Lipschitz maps in a nontrivial homotopy class is a nonempty geodesic lamination $$\Lambda_0(\rho) = \cap_{f}\stretch(f).$$ Contained in $\Lambda_0(\rho)$ is a maximal chain recurrent sublamination $$\Lambda(\rho) \subseteq \Lambda_0(\rho)$$ which we call the \textbf{cohomology stretch lamination} of $\rho$.

        We record these facts, proven in various forms by \cite{FathiSiconolfi}, \cite{GueritaudKassel}, and \cite{DaskalopoulosUhlenbeck}, as the following two theorems.

        \begin{thm}[Theorem 1.3 \cite{GueritaudKassel}] Let $M$ be a closed hyperbolic 3-manifold and $\rho\in H^1(M)$ a nontrivial cohomology class. The intersection $\cap_{f}\stretch(f)$ over all best Lipschitz maps in the homotopy class defined by $\rho$ is a nonempty geodesic lamination $\Lambda_0(\rho)$.
        \end{thm}

        \begin{thm}[Theorem 1.7 \cite{FathiSiconolfi}, Theorem 1.1 \cite{DaskalopoulosUhlenbeck}, Theorem 1.3 \cite{GueritaudKassel}]\label{thm:maps} Let $M$ be a closed hyperbolic 3-manifold and $\rho\in H^1(M)$ a nontrivial cohomology class. There exists a best Lipschitz map $f:M\to \R/\Z$ in the homotopy class determined by $\rho$ such that the set $\stretch(f)$ is a geodesic lamination. The map can be taken to be $C^1$ as in \cite{FathiSiconolfi}, with $\stretch(f)=\Lambda_0(\rho)$ as in \cite{GueritaudKassel}, or infinity harmonic as in \cite{DaskalopoulosUhlenbeck}.
        \end{thm}

        We also record here the corollary that the cohomology stretch lamination is orientable.

        \begin{prop}\label{prop:orientable} Cohomology stretch laminations are orientable.
        \end{prop}
        \begin{proof} We can identify $\Lambda(\rho)$ with a union of flowlines of a continuous vector field via the $C^1$ best Lipschitz map in \cite{FathiSiconolfi}.
        \end{proof}

        \section{Stretch functional and stable norms}\label{sec:stablenorm}

        Let $M$ be a finite volume hyperbolic 3-manifold possibly with cusps or boundary.
        Consider the real vector space $\cG(M)$ of weighted free homotopy classes of oriented closed curves with positive minimal length with the relations $\g^k=k\g$ and $-\g = \overline\g$, where $\overline \g$ denotes the curve $\g$ with reversed orientation. Note that when $M$ has cusps, peripheral curves are not included in $\cG(M)$.

        For a chain $$g=\sum r_i\gamma_i\in \cG(M),$$ Define $$|g|_M =\sum|r_i|\cdot|\gamma_i|_M,$$ where $|\gamma_i|_M$ is the minimal length of a curve in $M$ freely homotopic to $\gamma_i$. Let $\rho\in H^1(M)$ be a cohomology class. Then $\rho$ defines a map $\cG(M)\to\R$ by identifying elements in $\cG(M)$ with singular cycles.

        Given a Margulis constant $\mu\geq 0$, one can also consider $\cG(M_{\geq\mu})$.
        When $M$ has cusps, the peripheral curves \emph{are} included in $\cG(M_{\geq\mu})$ as there now exist length minimizing representatives. Via the inclusion $M_{\geq\mu}\into M$, the class $\rho$ pulls back to a class on $M_{\geq\mu}$, and then extends to a map $\cG(M_{\geq\mu})\to\R$.

        Let $K=K_{\rho,\mu}:\cG(M_{\geq\mu})-\{0\}\to\R$ be the function $$K(g) = \frac{\rho(g)}{|g|_{M_{\geq\mu}}}.$$ By convention, we set $K(0)=0$.
        Given $\mu\geq 0$, we define the \textbf{$\mu$-thick stable norm} $\norm{\cdot}_\mu$ on $H^1(M)$ as follows: $$||\rho||_{\mu} = \sup\limits_{g\in\cG(M_{\geq\mu})}K_{\rho,\mu}(g).$$

        For $M$ compact and $\mu=0$, this is the usual (dual) stable norm on $H^1(M)$; also called the comass (see Section 4 of \cite{Gromov} for a general reference).

        For any $f$ in the homotopy class defined by $\rho$ and any chain $\g\in\cG(M)$, we have $\rho(\g) = \int_\g df$, so we see $\norm{\rho}_0\leq L(\rho).$ Indeed, this turns out to be an equality and one can compute the stable norm as the supremum of $\rho(\sigma)/\len(\sigma)$ over simple closed curves $\sigma$.

        \begin{thm}[Theorem 5.8 \cite{DaskalopoulosUhlenbeck}] Let $M$ be a closed hyperbolic 3-manifold and $\rho\in H^1(M)$. Then the cohomology stretch and stable norm are equal: $L(\rho)=||\rho||_0$. Moreover, the stable norm is realized by a sequence of simple closed curves approximating a leaf of the stretch lamination.
        \end{thm}

        For $0$-surgery classes, there is a straight forward estimate relating the boundary slope length to the stable norm of the extended class.

        \begin{lem}\label{lem:nzbound} Let $W$ be a cusped hyperbolic 3-manifold.
            If $\rho\in H^1(W)$ is a $0$-surgery class with boundary slope $\bs$, then $$\norm{\rho_{\bs}}_0> \frac{n}{2\pi}(\ell^2-28.78),$$ where $\ell = \nlen(\bs)$ is the total normalized length of the slope $\bs$ and $n$ is the number of cusps.
        \end{lem}
        \begin{proof}
            Let $c$ be the unoriented multicurve of the Dehn filling cores in $M=W(\bs)$.
            For a given orientation on $c$, Poincaré duality gives that $\rho_{\bs}(c)$ is the algebraic intersection of $c$ and the Poincaré dual surface $F_\bs$
            The geometric intersection of each core curve $c_i$ with $F_\bs$ is 1. The multicurve can be oriented so that $\rho(c)=n$, where $n$ is the number of cusps. Quantitative Neumann-Zagier (Theorem \ref{thm:qNZ}) then gives an upper bound on the total length of the core curves in the Dehn filling, giving the estimate $||\rho||_0> \frac{n}{2\pi}(\ell^2-28.78)$.
        \end{proof}

        \section{Thick stable norm and Thurston norm}\label{sec:thickstablethurstoncomparison}

        For 3-manifolds, there is a natural topological norm, the \textbf{Thurston norm}, which is induced by Poincaré-Lefschetz duality $H^1(W)\cong H_2(W,\d W)$: $$\norm{\rho}_{Th}=\min\limits_{[S]\text{ dual to } \rho}\chi_-(S),$$
        where $\chi_-(S)$ is defined for connected surfaces as $\max\{-\chi(S),0\}$ and is extended to disconnected surfaces by summing over components.

        In this section, we relate the thick stable norm and the Thurston norm in Dehn surgery families.
        We follow the basic ideas of the proofs of norm comparisons in \cite{BD}; see also \cite{Hans} for complimentary results on $L^2$ harmonic forms in the cusped case. To do this, we need to know something about harmonic forms and least area surfaces in cusped hyperbolic 3-manifolds.

       The \textbf{harmonic norm} $$||\rho||_2^2=\inf\left\{\int_W\omega\wedge\star\omega~:~ \omega\in \Omega^1(W) \text{ representing $\rho$}\right\}$$
     is another natural geometric complexity measure for cohomology classes of Riemannian manifolds.  For closed Riemannian manifolds, this norm is realized by harmonic forms, where harmonic means $$d\omega=d^*\omega=0.$$ For compact Riemannian manifolds with boundary, this remains true when one imposes the Neumann boundary condition on $\omega$, which says $$i^*\star \omega = 0,$$ where $i:\d W\to W$ is the inclusion. More importantly for our application, the Neumann boundary condition facilitates Poincaré duality and enables a mean-value inequality.

     We require the following mean-value inequality of \cite{DiCerboStern}, slightly rephrased for our setting.
     \begin{thm}[Proposition 47 \cite{DiCerboStern}]\label{thm:DiCerboStern} Let $W$ be a cusped hyperbolic 3-manifold manifold and let $\mu$ be a Margulis constant that is smaller than $\sys(W)/4$. Let $h_\mu$ be a harmonic form satisfying Neumann boundary conditions for $W_{\geq\mu}$. Then there is a constant $C$ depending on $W$ and $\mu$ such that $||h_\mu||_\infty\leq C||h_\mu||_2$.
     \end{thm}

     \begin{proof} In Proposition 47 in \cite{DiCerboStern}, the manifold $M_T$ in their statement is obtained from drilling a closed manifold with negative curvature and has its induced metric, the argument only uses the drilled manifold however so applies to the manifold $W_{\geq\mu}$. As the boundary of $W_{\geq \mu}$ is smooth, one can avoid the issue of "modified tubes" in their Proposition 47. Alternatively, one can put $W_{\geq\mu}$ into their setting by taking a large Dehn filling of $W$ with the metric from the Gromov-Thurston $2\pi$-Theorem, then drilling.
     As stated in \cite{DiCerboStern}, the constant $C$ then depends on the Margulis constant and the second fundamental form of the boundary, which in turn only depends on $W$ and the Margulis constant.
     \end{proof}

        To relate the harmonic norm to the Thurston norm, we need the following theorem about least area surfaces.

        \begin{thm}[Theorem 4.4 \cite{HassRubinsteinWang}, \cite{Ruberman}, \cite{CollinHauswirthRosenberg}]\label{thm:HRW}
        Let $\overline W$ be a compact 3-manifold whose interior $W$ admits a complete hyperbolic metric of finite volume. Let $\overline S$ be an essential properly embedded
        surface in $\overline W$ of finite type.
        Then $S := \interior(\overline S)$ is properly homotopic in $W$ to a surface $S'$ with
        least area in its homotopy class. Moreover, the area of $S'$ is bounded by $2\pi\chi_-(S)$.
        \end{thm}

        The existence statement is the statement appearing in Theorem 4.4 \cite{HassRubinsteinWang}, see also \cite{Ruberman} or \cite{CollinHauswirthRosenberg} for the theorem with isotopy replacing homotopy. The area estimate follows for instance from the finite total curvature theorem of Collin-Hauswirth-Rosenberg and the upper curvature bound given by the curvature of hyperbolic space. See also Remark\nopagebreak~3 and the remark after the proof of Theorem 10 in \cite{CollinHauswirthMazetRosenberg} for two-sided area bounds.

        The proof of Theorem \ref{thm:chainnormineqs} below includes an analogue of the upper bound in the main theorems of \cite{BD} and \cite{Hans} for relative cycles, though unfortunately with a non-explicit constant and with additional dependence on cusp shapes and the Margulis constant used to truncate the cusped manifold.
        In \cite{Hans}, Han shows that the harmonic representative in the cusped manifold of a class dual to a relative cycle with nonempty boundary is not in $L^2$, so some kind of thin-part truncation is needed for any $L^2$ norm comparison. See also the examples with vanishing injectivity radii in Theorem 1.4 of \cite{BD} and the analysis there; these examples also make an appearance here in Theorem\nopagebreak~\ref{thm:fiberedexamples}.

        \begin{thm}\label{thm:chainnormineqs}
            Let $W$ be a cusped hyperbolic 3-manifold and let $\rho\in H^1(W)$. Let $\mu< \sys(W)/4$~ be a Margulis constant.
            Then there is a constant $C$ depending only on $W$ and $\mu$ such that
            $\norm{\rho}_\mu\leq C\norm{\rho}_{Th}.$
        \end{thm}

        \begin{proof}
            The inequality is trivial if $\rho=0$ so we assume it is not.
            Let $h_{\mu}$ be the harmonic representative of $\rho$ with Neumann boundary condition for $W_{\geq \mu}$ and let $S'$ be the least area surface in Theorem \ref{thm:HRW} associated to a Thurston norm minimizing surface $S$ dual to $\rho$. Let $S'_\mu=S'\cap W_{\geq\mu}$; this represents the Poincaré dual of the class $\rho$ pulled back to $W_\geq\mu$. We denote by $\norm{\cdot}_{t,p}$ the $L^p$ norm of forms on the submanifold $W_{\geq t}\subset W$.

            First we follow the argument in Equation 5.2 in \cite{BD} with the pointwise estimate from Theorem \ref{thm:DiCerboStern} to obtain a bound on the harmonic norm by the Thurston norm:

            \begin{align*}
                \norm{h_\mu}^2_{\mu,2} &= \int_{W_{\geq\mu}}h_{\mu}\wedge \star h_{\mu} \\
                             &= \int_{S_{\mu}'} \star h_{\mu} \text{, by Poincaré duality via the Neumann boundary condition}\\
                             &\leq \norm{h_{\mu}}_{\mu,\infty}\area(S') \text{, using that $|\star h_\mu|=|h_\mu|$ and $\area(S'_\mu)\leq\area(S')$}\\
                             &\leq 2\pi\norm{h_{\mu}}_{\mu,\infty}||\rho||_{Th} \text{, by the area-Euler characteristic estimate for $S'$}\\
                             &\leq 2C\pi\norm{h_{\mu}}_{\mu,2}||\rho||_{Th} \text{, using Theorem \ref{thm:DiCerboStern}.}
            \end{align*}
            Dividing by $\norm{h_\mu}_{\mu,2}$ gives $$\norm{h}_{\mu,2}\leq 2C\pi||\rho||_{Th}.$$

            Then we have $$\norm{\rho}_\mu\leq \norm{h_\mu}_{\mu,\infty}$$ because for any $\g\in\cG(W_{\geq\mu})$, the averages satisfy
            $$\rho(\g)/\len(\g) = 1/\len(\g)\int_\g h_\mu\leq ||h_\mu||_{\mu,\infty}.$$
            Using again the pointwise bound from Theorem \ref{thm:DiCerboStern}, we combine our estimates and constants to obtain $$\norm{\rho}_\mu\leq  C||\rho||_{Th}.$$
        \end{proof}
        Using the Drilling Theorem, we can transport this to the thick stable norm in large Dehn fillings.

        \begin{lem}\label{lem:drillnorm} Let $W$ be a cusped hyperbolic 3-manifold and $\mu<\min\{\sys(W)/4,\log(3)/1.2\}$ a Margulis constant. There are constants $C,L>0$ depending on $W$ and $\mu$ such that if $\rho\in H^1(W)$ is a surgery class with compatible slope $\bs$ satisfying $\nlen(\bs)>L$, then $||\rho_\bs||_{\mu}\leq C||\rho||_{\mu}.$
        \end{lem}

        \begin{proof} The Lipschitz embedding in the Drilling Theorem (Theorem \ref{thm:drill}) with $J=2$ and $\e=1.2\mu$ gives for sufficiently large Dehn filling slopes the estimate between the thick stable norm of $||\cdot||_{\mu}$ in $W(\bs)$ and the thick stable norm $||\cdot||_{1.2\mu}$ in $W$ that $||\cdot||_{\mu}\leq2||\cdot||_{1.2\mu}$ as the pullback map induces an isomorphism on cohomology and the map itself can only change lengths by at most a factor of 2. 
            One can also compare the thick stable norms $||\cdot||_{1.2\mu}$ and $||\cdot||_{\mu}$ in $W$; this comparison depends only on $W$ and $\mu$.  Because $\rho_\bs$ extends $\rho$, the restriction of $\rho_\bs$ to the thick part of $W(\bs)$ is mapped to $\rho$ restricted to the thick part of $W$ by the pullback of the map in the Drilling Theorem, so indeed we have that there is a constant $C$ depending only on $W$ and $\mu$ such that $\norm{\rho_\bs}_\mu\leq C\norm{\rho}_\mu$.
        \end{proof}

        Theorem \ref{thm:thickstableThurston} is then an easy corollary.

        \thickstableThurston*
        \begin{proof} Combine Lemma \ref{lem:drillnorm} and Theorem \ref{thm:chainnormineqs} and the constants therein.
        \end{proof}

    \section{Some tube estimates}\label{sec:bigtubes}

        Let $W$ be a hyperbolic 3-manifold and let $\cT$ be an embedded tube in $W$ centered around a core geodesic.
         The tube $\cT$ can be given cylindrical coordinates about its core $(r, \theta, z) \in [0, \radius(\cT)]\times
         [0, 2\pi] \times [0, \e]$ with the identification $(r, \theta, \e) \equiv (r, \theta + \theta_0, 0)$ for some twist angle $\theta$; the metric on $\cT$ then takes the form
         $dr^2 + \sinh^2(r )d\theta^2 + \cosh^2(r )dz^2.$

         Let $\{T_t\}$ be the foliation of $\cT-\core(\cT)$ by Euclidean tori. For $0 < s<t<\radius(\cT)$, let $\cT_{s,t}$ be the union of leaves $T_r$ for $r\in[s,t]$. There are cylindrical projection maps $\fp_{s,t}:\cT_{s,t}\to T_s$. The projection map $\fp_{s,t}$ fixes $T_s$ and is a 1-Lipschitz homotopy equivalence. There is also an outward projection $\fp_{t,s}:\cT_{s,t}\to T_t$ which expands lengths and fixes $T_t$. The composition of inward and outward projections is the identity map $T_s\to T_s$.

         We denote by $\len_t$ the length function in the flat torus $T_t\subset\cT$.

    \begin{lem}\label{lem:lengthprojectionbound}

    Let $0<r<R\leq\radius(\cT)$ and let $T_r$ and $T_R$ be the pair of flat tori in the tube $\cT$ distance $r$ and $R$ respectively from the core $\core(\cT)$. Let $c_R$ be a rectifiable curve in $T_R$. Let $\fp_{R,r}$ be the (cylindrical) projection from $T_R$ inwards to $T_r$. Then $$\len_{r}(\fp_{R,r}(c_R))\leq(e^{-r}+e^{r-R})\len_{R}(c_R).$$

    \end{lem}
    \begin{proof}Lemma 8.3 in \cite{FuterPurcellSchleimerTransactions} gives for rectifiable curves $c_r$ in $T_r$ and the \emph{outward} cylindrical projection the bound $$ \frac{\cosh(R)}{\cosh(r)}\leq\frac{\len_{R}(\fp_{r,R}(c_r))}{\len_{r}(c_r)} \leq  \frac{\sinh(R)}{\sinh(r)}.$$ For the inward projection, this implies
        $$\len_{r}(\fp_{R,r}(c_R)\leq\frac{\cosh(r)}{\cosh(R)} \len_{R}(c_R).$$

        Next observe:
           $$ \frac{\cosh(r)}{\cosh(R)} =\frac{e^{R-r}}{e^R+e^{-R}} + \frac{e^{r-R}}{e^R+e^{-R}} = e^{-r}(\frac{e^R}{e^R+e^{-R}})+e^{r-R}(\frac{1}{e^R+e^{-R}})\\
                  \leq e^{-r}+e^{r-R}.$$
    \end{proof}

    Let $\sigma$ be a union of arcs in $W$ and let $\cT$ be a tube in $W$.
    Define the \textbf{tube depth} $\delta_\cT(\sigma)$ to be $\radius(\cT)-t$, where $T_t$ is the torus closest to $\core(\cT)$ in $\cT$ such that $\sigma \cap T_t\neq\emptyset$.

    \begin{lem}

        Fix $D\geq 2\log 4$ and let $\cT$ be a hyperbolic tube with radius $R\geq D+\log 4$; denote the boundary torus of $\cT$  by $T_R$.
        Let $\sigma$ be an arc in $\cT$ from $T_R$ to itself. Let $\bsigma$ be the loop obtained by attaching to $\sigma$ the Euclidean geodesic $\a$ in $T_R$ between its endpoints. Then if the depth of $\sigma$ satisfies $\delta_\cT(\sigma)>D$, the geodesic loop $\gamma$ homotopic to $\bsigma$ has length satisfying $$\len(\g) + D/4 +\len(\a)/2\leq \len(\overline\sigma).$$
    \end{lem}

    \begin{proof} From Lemma \ref{lem:lengthprojectionbound}, we get for a Euclidean arc $\alpha$ distance $R$ from the core, that if we project inwards to the torus at radius $R-D/2$, then if $R>2\log4 + D$ and $D>2\log 4$, we have $$\len_{R-D/2}(\fp_{R,R-D/2}(\alpha))\leq 1/2\len_{R}(\alpha).$$

        Let $\bsigma_1$ be the curve obtained by projecting the portion of $\bsigma$ in $\cT_{R,R-D/2}$ onto $T_{R-D/2}$, then appending the remaining portion of $\sigma$ contained in $\cT_{R-D/2,R-D}$. Because the projection fixes $T_{D/2}$, this gives a loop in the same homotopy class as $\bsigma$. Because these projections are $1$-Lipschitz, this can only decrease the length.

        Because $\sigma$ has depth greater than $D/2$, contained in $\sigma\cap\cT_{R,R-D/2}$ are a pair of distinct arcs $\beta_0$ and $\beta_1$, one running from $T_{R}$ to $T_{R-D/2}$ and the other running from $T_{R-D/2}$ to $T_R$, the total length of these arcs is at least $D$. Let $\beta = \beta_0\cup\beta_1$. The projection $\fp_{R,R-D/2}(\beta)$ is a sub-multicurve of $\bsigma_1$. If $\len_{R-D/2}(\fp_{R,R-D/2}(\beta))\leq D/2$, then we have saved $D/2$ in length.

        The projection $\fp_{R,R-D}$ factors as a composition of the projections $\fp_{R,R-D/2}$ and $\fp_{R-D/2,R-D}$. If we have separately saved length $\len(\alpha)/2$ and $D/2$ in the first projection, we can then estimate
        $$\len_{{R-D/2}}(\bsigma_1)+\len_{R}(\alpha)/2 + D/2 \leq \len(\bsigma),$$ which implies the lemma.

        We therefore assume  that $\len_{{R-D/2}}(\fp_{R,R-D/2}(\beta))\geq D/2$.

        In this case, we can find subarcs of $\bsigma_1$ with union $\beta'$ such that $\len_{R-D/2}(\beta')\geq D/2$. We project $\bsigma_1$ to $T_{R-D}$.  The image of $\beta'$ under the projection $\fp_{R-D/2,R-D}$ has length we can estimate as above using Lemma \ref{lem:lengthprojectionbound}:

        $$\len_{R-D}(\fp_{R-D/2,R-D}(\beta'))\leq (e^{D-R} + e^{-D/2})\len_{{R-D/2}}(\beta').$$ If $D\geq2\log4$ and $R\geq D+\log 4$, then $e^{D-R} + e^{-D/2}\leq 1/2$, so we have decreased the length of
        $\beta'$ by at least a factor of 2, and therefore save length at least $D/4$.
        Let $\bsigma_2$ be the loop obtained by appending to $\fp_{0,D}(\bsigma_1)$ the portion of $\bsigma$ contained in $\cT_{D,d}$.
        Again, the projection $\fp_{R,R-D}$ factors as a composition of the projections $\fp_{R,R-D/2}$ and $\fp_{R-D/2,R-D}$. We have saved separately length $\len(\alpha)/2$ and $D/4$, so in total we can estimate
        $$\len_D(\bsigma_2)+\len_{R}(\alpha)/2 + D/4\leq \len(\bsigma),$$ which implies the lemma.
    \end{proof}

    This immediately gives the following useful lemma.

    \begin{lem}\label{lem:tubearcdepth}
        Let $D>2\log 4$.
        Let $M$ be a closed hyperbolic 3-manifold containing a tube $\cT$ whose boundary $\d\cT$ satisfies $8\diam(\d\cT)\leq D$ and which has tube radius $\radius(\cT)>D+\log 4$.
        Let $\sigma$ be an arc in $\cT$ from $\d\cT$ to itself of depth $\delta_\cT(\sigma)>D$. Let $\bsigma$ be the loop obtained by attaching to $\sigma$ the Euclidean geodesic $\a$ in $\d \cT$ between its endpoints. Then the geodesic loop $\gamma$ homotopic to $\bsigma$ has length satisfying $$\len(\g) + 2\diam(\d\cT) < \len(\bar\sigma).$$
    \end{lem}

    Next we show that one can replace closed geodesics traveling sufficiently far into tubes with homologous shorter multicurves. This is the main result from this section that we require for our theorems. This shortening process increases the value of the functional $K=K_{\rho,\mu}$ from Section \ref{sec:stablenorm}, so cannot be done to loops realizing the best Lipschitz constant; in particular, closed leaves in cohomology stretch laminations do not travel far into Margulis tubes. For loops approximating non-closed leaves of the best Lipschitz lamination, it still proves useful.

    \begin{figure}[h]
        \centering
        \includegraphics[scale=.15]{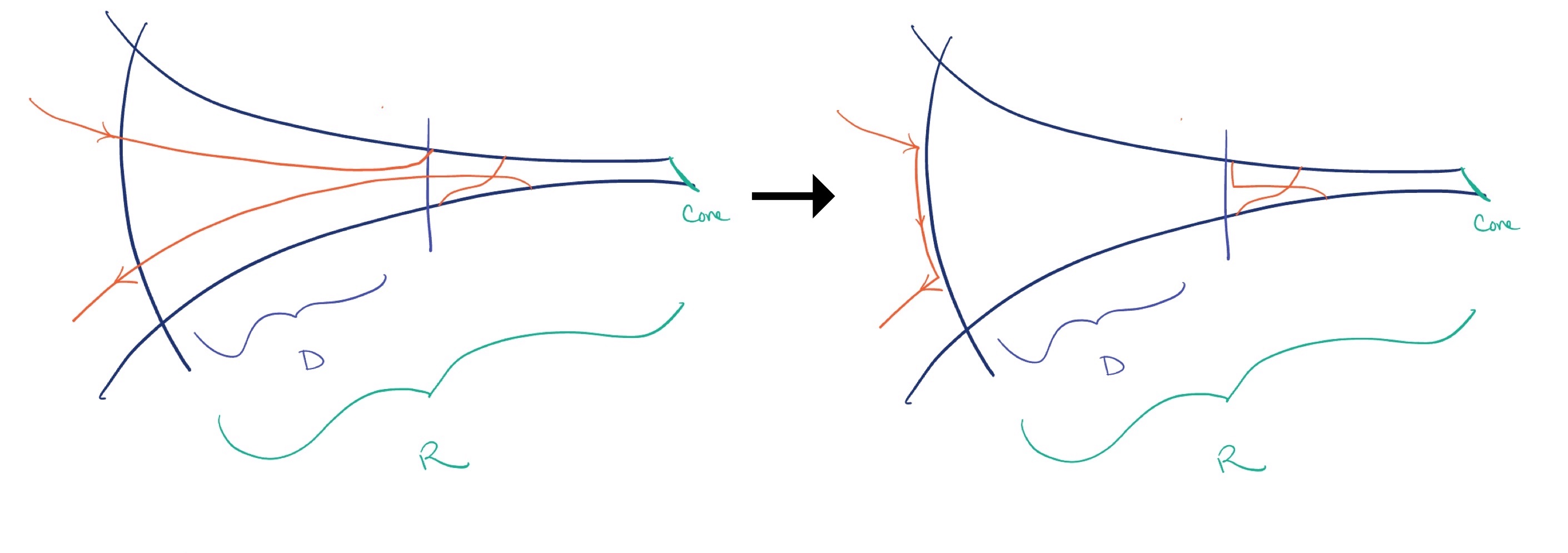}\caption{Cartoon of curve entering tube and the shorter multicurve obtained in Lemma \ref{lem:shortmulticurve}.}
    \end{figure}

    \begin{lem}\label{lem:shortmulticurve}

        Let $M$ be a closed hyperbolic 3-manifold and $\cT$ a tube in $M$. Let $D \geq 8\diam(\d \cT)+2\log4$ and assume the tube has radius $R>D+\log 4$. Let $\g$ be a closed curve in $M$ with $\delta_\cT(\g)>D$. Then there is a multicurve $g$ homologous to $\g$ and with $\len(g)<\len(\g)$ such that every component of the geodesic tightening of $g$ has tube depth less than $D$ or is contained in $\cT$. Moreover, one can assume $g\cap (M-\cT) = \g\cap (M-\cT)$.
    \end{lem}
    \begin{proof}
        Cut $\g$ along $\d\cT$. Because $\d\cT$ is nullhomologous, there are an even number, say $2n$, of intersections which come in oppositely oriented pairs. The pairs of intersections come from strands $\sigma_i$ contained in $\cT$ as well as from complimentary strands $\eta_i$ contained in $M-\cT$. Create closed curves $\bsigma_i$ and $\overline\eta_i$ by attaching Euclidean geodesics in $\d \cT$ between the endpoints of the strands, matching interior strands with oppositely oriented interior strands and likewise exterior strands with oppositely oriented exterior strands, to get a multicurve each component of which is contained entirely in either the thick or thin part of $M$. This process adds in total no more length than $2n\diam \d\cT$. By Lemma \ref{lem:tubearcdepth}, the curves $\bsigma_i$ inside the tubes can be geodesically tightened in $\cT$, reducing their length by strictly more than $2\diam(\d\cT)$, so in total we save strictly more length than $2n\diam\d\cT$; we now denote by $\bsigma_i$ these tightened curves. It follows that the multicurve $\sum\overline\eta_i+\sum\bsigma_i$ has length strictly less than $\g$ and that in $M-\cT$, the multicurve restricts to a union of subarcs of $\g$, because we only needed to tighten the components in the tube to get the needed length savings.

        The procedure used to produce the multicurve does not change the homology class.
        When one geodesically tightens the curves $\overline\eta_i$, because $D>\diam\d\cT$, none of the curves travel further than depth $\diam(\d\cT)$ into the tube unless it is homotopic to a multiple of the core, as that would contradict length minimality.
    \end{proof}

    \section{Stretch laminations near deep tubes}\label{sec:stretchtubes}

    We now apply results from the previous section to stretch laminations in particular. We begin with a simple application to closed leaves in stretch laminations.

    \begin{prop}\label{prop:thinclosedstretch} Let $M$ be a closed hyperbolic 3-manifold. Let $\mu$ be a $(D,\log 4)$-deep Margulis constant for $M$ such that every maximal tube intersecting the $\mu$-thin part of $M$ has boundary torus with diameter less than $D/8-1/4\log 4$. Then if $f:M\to \R/\Z$ is a nontrivial best Lipschitz map and $\lambda\subset\stretch(f)$ is a closed geodesic, then $\lambda$ is either contained $M_{\leq\mu}$ or $M_{\geq \mu}$. In particular, $\lambda$ does not travel from the thick part into the thin part.
    \end{prop}
    \begin{proof} Let $K(\g)=\rho(\g)/|\g|_M$, where $\rho$ is the cohomology class associated to $f$.  Because $\lambda\subset\stretch(f)$ and $\rho$ is nontrivial, we have $K(\g)= \rho(\lambda)/|\lambda|_M = ||\rho||_0>0$. Because $\lambda$ travels from the thick part to the thin part, it must have tube depth at least $D$. The tube satisfies the conditions in Lemma \ref{lem:shortmulticurve}, so there is a multicurve $g$ with $\norm{\rho}_0\geq K(g)>K(\lambda)=\norm{\rho}_0$, a contradiction.
    \end{proof}

     Next we show a condition on the stable norm forces core curves in Margulis tubes to be part of the cohomology stretch lamination.

    \begin{prop}\label{prop:lengthbound}        Let $M$ be a closed hyperbolic 3-manifold and let $\rho\in H^1(M)$ be a nontrivial cohomology class with associated cohomology stretch lamination $\Lambda(\rho)$. Let $\mu$ be a $(D,\log 4)$-deep Margulis constant for $M$.  For the collection of maximal tubes $\{\cT\}$ in $M$ intersecting the $\mu$-thin part of $M$, assume $D \geq 8\diam(\d \cT)+2\log4$.
        Then either $\norm{\rho}_0=\norm{\rho}_\mu$ or the stretch lamination contains some nonempty union of core curves of Margulis tubes in the $\mu$-thin part.

    \end{prop}
    \begin{proof}

        Let $f$ be a $C^1$ best Lipschitz map with stretch set a geodesic lamination and let $\omega=df$ be the corresponding 1-form; such a map exists by Theorem \ref{thm:maps} and, as noted in Proposition \ref{prop:orientable}, orients $\Lambda(\rho)\subset \stretch(f)$.

        Let $p\in\Lambda(\rho)$; then by chain recurrence, there is a sequence $\gamma_k$ of closed curves locally $C^1$-converging to oriented segments of $\Lambda(\rho)$ which go through $p$. 
        We orient these curves so that $\omega(\dot\gamma_k(x_k))$ converges pointwise to the best Lipschitz constant $||\rho||_0$ for any sequence $x_k\in\g_k$. Because $\Lambda(\rho)$ is oriented by the continuous form $\omega$, this can be done even when the leaf of $\Lambda(\rho)$ that $\g_k$ is following changes. This then implies $K(\g_k)\to \norm{\rho}_0$.

        By applying Lemma \ref{lem:shortmulticurve} for each tube into which $\g_k$ travels distance at least $D$ in $\cT_m$, we can replace each $\gamma_k$ with a geodesic multicurve $g_k$ homologous to $\gamma_k$, with $\len(g_k)\leq \len(\gamma_k)$, and such that each component of the geodesic tightening is contained in either the thick part or thin part. The multicurve $g_k$ satisfies $K(g_k)\geq K(\gamma_k)$ because $\len(g_k)\leq\len(\gamma_k)$ and $\rho(g_k)=\rho(\gamma_k)$.
        This implies that either the stable norm is realized in the thick part, or a Dehn filling core is part of the stretch lamination. To see this, choose the component $g_k^{\max}\subset g_k$ maximizing $K(g^{i}_k)$ among all components $g_k^i$ of $g$. By construction, the geodesic tightening of $g_k^{\max}$ is contained entirely in the thick or thin part and the sequence of such loops satisfies $K(g_k)\leq K(g_k^{\max})$, so $K(g_k^{\max})\to\norm{\rho}_0$. If the tightening of $g_k^{\max}$ is contained in the thin part for infinitely many $k$, then the core curve of a Margulis tube is part of the cohomology stretch lamination as every nontrivial closed curve in the thin part of a tube is homotopic to a power of the core. If it is not contained in the thin part for infinitely many $k$, then the stable norm is realized in the thick part. In particular, as soon as the thick part does \emph{not} realize the stable norm, we know some core curve of a Maruglis tube in the thin part is contained in the stretch lamination.
    \end{proof}

    Next is our main result about the leaves of stretch laminations near large tubes; this result says a stronger condition on the stable norm on a larger tube implies \emph{every} leaf of the stretch lamination is the core of a Margulis tube. Note that the depth assumption in Theorem \ref{thm:onlycores} differs by a factor of two from the depth assumption in Proposition \ref{prop:lengthbound}.

    \begin{thm}\label{thm:onlycores}        Let $M$ be a closed hyperbolic 3-manifold and let $\rho\in H^1(M)$ be a nontrivial cohomology class with associated cohomology stretch lamination $\Lambda(\rho)$. Let $\mu$ be a $(2D,\log 4)$-deep Margulis constant for $M$.  For the collection of maximal tubes $\{\cT\}$ in $M$, where each tube $\cT$ intersects the $\mu$-thin part of $M$, assume $D \geq 8\diam(\d \cT)+2\log4$ for each tube in the collection. If$~||\rho||_0> 3\norm{\rho}_\mu$, then every leaf of $\Lambda(\rho)$ is a core curve of a Margulis tube in the $\mu$-thin part.
    \end{thm}
    \begin{proof}
        If $\Lambda(\rho)\cap M_{\geq\mu}=\emptyset$, then the  cohomology stretch lamination is a union of core curves of Margulis tubes in the $\mu$-thin part, so the conclusion holds. We therefore assume otherwise.
        We now prove that if there is a point of $\Lambda(\rho)$ in the thick part, then $\norm{\rho}_0\leq 3\norm{\rho}_\mu.$
        Let $f$ be a $C^1$ best Lipschitz map with stretch set a geodesic lamination and let $\omega=df$ be the corresponding 1-form; such a map exists by Theoreom \ref{thm:maps} and, as noted in Proposition \ref{prop:orientable}, orients $\Lambda(\rho)\subset \stretch(f)$.

        Let $p\in\Lambda(\rho)$ be a point in the $\mu$-thick part of $M$. By chain recurrence, there is a sequence $\gamma_k$ of closed curves locally $C^1$-converging to oriented segments of $\Lambda(\rho)$ which go through $p$. 
        We orient these curves so that $\omega(\dot\gamma_k(x_k))$ converges pointwise to the best Lipschitz constant $||\rho||_0$ for any sequence $x_k\in\g_k$. Because $\Lambda(\rho)$ is oriented by the continuous form $\omega$, this can be done even when the leaf of $\Lambda(\rho)$ that $\g_k$ is following changes. This then implies $K(\g_k)\to \norm{\rho}_0$.
        If these loops stay in the thick part of $M$, then the stable norm is realized in the thick part, so $\norm{\rho}_0=\norm{\rho}_\mu\leq 3\norm{\rho}_\mu$. We can therefore assume (for large $k$), the loops $\g_k$ travel into the thin part.

        Let $\cT_m^{\max}$ be the collection of maximal tubes in $M$ intersecting $M_{\leq\mu}$. Let $\cT_{m}$ be the smaller tubes whose boundaries are at depth $D$ inside $\cT_m^{\max}$. Note that the diameter of the boundary of $\cT_m$ is smaller than the diameter of $\cT_m^{\max}$ and that the conditions of Lemma \ref{lem:shortmulticurve} hold for these smaller tubes. Because $\g_k$ intersects the thin part, it must travel more than distance $D$ into some tube $\cT_m$.
        By applying Lemma \ref{lem:shortmulticurve} for each tube $\cT_m$ into which $\g_k$ travels more than distance $D$, we can replace each $\gamma_k$ with a multicurve $g_k$ homologous to $\gamma_k$ with $\len(g_k)\leq \len(\gamma_k)$ which agrees with $\g_k$ in $M-\cup\cT_m$. 

        \begin{figure}[h]
            \centering
            \includegraphics[scale=.32]{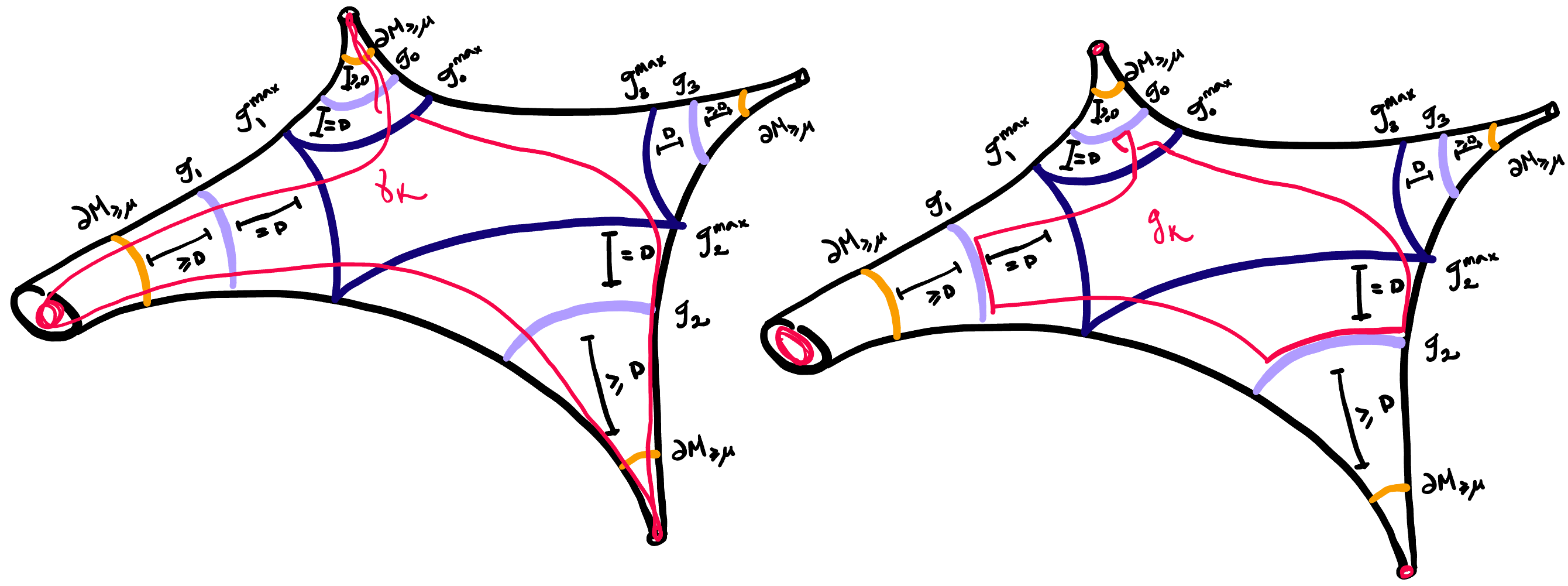}\caption{Cartoon of tubes and approximating curves in application of Lemma \ref{lem:shortmulticurve} in Theorem\nopagebreak~\ref{thm:onlycores}}\label{fig:onlycores}
        \end{figure}

        Let $\sigma_k^i$ be the sub-strands of $\gamma_k$ that are contained in $M-\cup_m\cT_m$ and let $g_k^j$ be the components of $g_k$ obtained from some union of strands $\sigma_k^i$ by attaching the Euclidean geodesic arcs between endpoints of strands on the boundaries of tubes $\cT_m$. By the depth assumption on the Margulis constant, these loops lie in the thick part. Note that if any strand intersects the boundary $\d \cT_m$ twice without exiting the tube $\cT^{\max}_m$, it cannot approximate a geodesic segment, so we assume there are no such strands. We can index and pair each strand $\sigma_k^i$ with the arc that follows it so that $\sigma_k^i$ is followed by the arc $\alpha_k^i$, the endpoint of $\alpha_k^i$ is at the beginning of the strand $\sigma_k^{i+1}$, which is then followed by the arc $\alpha_k^{i+1}$ and so on. Fix some component $g_k^j$ and assume the strands are indexed so that $g_k^j = \cup_i (\sigma_k^i\cup\alpha_i)$.

        Because the oriented loops $\g_k$ are locally $C^1$-converging to oriented segments of $\Lambda(\rho)$, there is a sequence $\e_k>0$ with $\e_k\to0$ with $k$ such that for each strand $\sigma_k^i$, there is a lower bound $$\int_{\sigma_k^i}\omega > (||\rho||_0-\e_k)\len(\sigma_k^i).$$
        Now we estimate
        \begin{align*}
            \rho(g_k^j)&= \sum_i \int_{\sigma_k^i}\omega+\int_{\alpha_k^i}\omega \\
                              &> \sum_i \left(\left(\norm{\rho}_0-\e_k\right)\len(\sigma_k^i) + \int_{\alpha_k^i}\omega\right)\\
                              &\geq \sum_i \left(\left(||\rho||_{0} - \e_k\right)\len(\sigma_k^i) -||\rho||_{0}D\right)\\
                              &= \sum_i \left(||\rho||_{0}\left(\len(\sigma_k^i) -D\right)- \e_k\len(\sigma_k^i) \right)
        \end{align*}

        where we use that $||\omega||_{\infty}=\norm{\rho}_0$ and $\len(\a_k^i)\leq D$ to bound the possible negative contribution coming from integrating $\omega$ over the arc $\alpha_k^i$.

        Next observe that the depth assumption on the tubes $\cT_m$ implies each strand $\sigma_k^i$ must travel more than distance $2D$, because it must travel from the boundary of $\cT_m$ out of a maximal tube $\cT_m^{\max}$, which requires length $D$, then it must travel into (possibly the same) maximal tube to the boundary of the corresponding smaller tube, which again requires distance $D$. This implies $\len(\sigma_k^i)\geq 2D$ and thus $$\len(\sigma_k^i)-D\geq\len(\sigma_k^i)/2.$$

        Applying this to the estimate above, we get

        $$ \rho(g_k^j) \geq \sum_i \left(||\rho||_{0}\len(\sigma_k^i)/2 - \e_k\len(\sigma_k^i)\right) = \frac{||\rho||_{0}-2\e_k}{2}\sum_i \len(\sigma_k^i).$$

                Dividing by length and noting that $$|g_k^j|_{M\geq\mu}\leq \sum_i D+\len(\sigma_k^i)\leq \frac{3}{2}\sum_i\len(\sigma_k^i)$$ we estimate

               $$ \norm{\rho}_\mu \geq \frac{\rho(g_k^j)}{|g_k^j|_{M\geq\mu}} \geq \frac{||\rho||_{0}-2\e_k}{3\sum_i\len(\sigma_k^i)}\sum_i\len(\sigma_k^i) = \norm{\rho}_0/3-2\e_k/3.$$
               This holds for \emph{every} component $g_k^j$ of $g_k$ in $M_{\geq\mu}$, so we can choose such a component for each $k$ and take the limit as $k\to\infty$ to obtain $3||\rho||_\mu\geq\norm{\rho}_0.$
    \end{proof}


    \section{Stretch laminations in surgery families}\label{sec:mainthms}
    In this section we combine our previous results and obtain our main theorems.
    The main idea is that in a Dehn surgery family, one can apply Theorem \ref{thm:onlycores} uniformly for a sequence of large fillings. Then using our norm comparisons, we can replace the stable norm condition on the filled manifold with a condition on slope lengths and the Thurston norm in the original cusped manifold.

    \laminationsintubes*
    \begin{proof}

    Using the diameter bound in Lemma \ref{lem:diambound}, we can fix a single constant $D$ such that $D\geq8\diam(\d \cT)+2\log 4$ for each tube $\cT$ in a collection of maximal tubes around the Dehn filling core curves for every hyperbolic filling of $W$. 
    By Proposition \ref{prop:deepmargulis}, there is a constant $L>0$ and a Margulis constant $\mu$ that is $L$-persistantly $(2D,\log 4)$-deep. Furthermore, we can additionally ensure the $L$-persistantly $(2D,\log 4)$-Margulis constant is smaller than $\min\{\sys(W)/4,\log(3)/1.2\}$.
    For the filled manifolds $W(\bs)$ for all slopes $\bs$ with $\nlen(\bs)>L$,  we are in the situation of Theorem \ref{thm:onlycores} for this Margulis constant $\mu$, so the conclusion follows. 
    \end{proof}

    The following uses the thick stable norm-Thurston norm comparison of Theorem \ref{thm:chainnormineqs} to give a condition relating the Thurston norm and filling slope length that ensures a sequence of $0$-surgery classes extend to classes whose cohomology stretch laminations are (eventually) a union of Dehn filling cores.

    \littleoh*

    \begin{proof}
    Set $\ell_i=\nlen(\bs_i)$.
    Let $\Lambda_i$ be the cohomology stretch lamination of $\rho_i$. Let $L$ and $\mu$ be the constants in Theorem \ref{thm:laminationsintubes}. With this Margulis constant, all the tubes intersecting the thin part in Dehn fillings with complete slopes $\bs$ with length greater than $L$ come from the Dehn filling.

    By Theorem \ref{thm:thickstableThurston}, after possibly increasing $L$, there is a constant $C$ depending on $W$ and $\mu$, such that if $\ell_i>L$ then $$||\rho_i||_\mu\leq C||\rho_i^W||_{Th}.$$

    Lemma \ref{lem:nzbound} then gives the stable norm estimate $$||\rho_i||_0 \geq \frac{n}{2\pi}(\ell_i^2-28.78),$$ where $n$ is the number of cusps.
    The condition on the growth of the Thurston norm implies there is some $N$ such that for all $i>N$ $$\frac{n}{2\pi}(\ell_i^2-28.78)> 3C||\rho_i^W||_{Th}.$$ Combining these, we have for large $i$ the estimate $$||\rho_i||_0 >  3||\rho_i||_\mu.$$

    Therefore, the stable norm condition in Theorem \ref{thm:laminationsintubes} does not hold for large $i$, so the cohomology stretch lamination must be a union of Dehn filling cores.
    \end{proof}

    We can apply the Theorem \ref{thm:littleoh} to an explicit family of Dehn fillings constructed by Brock and Dunfield \cite{BD}. For us, the point of this construction is that there are infinitely many $0$-surgery fibered classes and one can directly compare the Thurston norm of the classes with the lengths of their boundary slopes. This immediately gives the little oh condition in Theorem \ref{thm:littleoh}. These examples enjoy extra symmetry which allows us to completely identify the stretch laminations.
    By Proposition 3.5 and Remark 3.6. in \cite{FarreLandesbergMinsky}, this then completely identifies the projection to the compact base of the $\omega$-limit set $\cQ_\omega$ of quasiminimizing points in the infinite cyclic covers associated to these fibrations.

    \fiberedexamples*

    \begin{proof}

        We recall the construction of Brock and Dunfield given in \cite{BD}, Theorem 1.4.
        Let $W$ be a 2-cusped hyperbolic 3-manifold manifold fibering over the circle, such that the boundary torus inclusions induce isomorphisms on first homology:
         $$H_1(T_j ; \Z) \to H_1(W ; \Z).$$ Additionally, assume there is an orientation reversing involution that exchanges the cusps and acts on $H_1(W;\Z)$ as the identity. An explicit link exterior found by Brock and Dunfield with these properties is shown in Figure \ref{fig:link}. Since
        $W$ fibers and has Betti number 2, there is a 1-dimensional face of the Thurston norm ball such that every integral class in the interior of the cone $\cC$ over this face is a fibration.

        Fix integral classes $\alpha$ and $\beta$ interior to $\cC$ generating $H^1(W)$. There are dual curves $a$ and $b$ generating $H_1(W)$ with $\alpha(a)=\beta(b)=1$ and $\alpha(b)=\beta(a)=0$. Let $M_n$ be the Dehn filling of $ W$ with slope $\bs_n = (s_j^n)$ with the curve on the boundary tori $T_j$ homologically
        equal to $s_j^n = a - nb$. The cohomology $H^1(M_n ; \Z)$ is $\Z$ and
        is generated by the extension $\rho_n$ of the class $\rho_n^W = \alpha -n\beta$ to $M_n$. The Thurston norm is linear in the cone over the face $\cC$, so the Thurston norm of $\rho_n^W$ is exactly $n||\alpha||_{Th}+||\beta||_{Th}$.  The filling slope $\bs_n= (a-nb,a-nb)$ has total normalized length comparable to $n$, because the lengths of Euclidean geodesics on the cusp tori are comparable to the Euclidean $\ell^2$ norms of their homology classes. In particular, since the Euclidean $\ell^2$ norm and Thurston norm can be compared, the Thurston norm is $o\left(\nlen(s_n)^2\right)$.
        The homological conditions imply all of these slopes are $0$-surgery slopes. The filled manifolds also fiber via the capped off surfaces.

        Thus, the conditions of Theorem \ref{thm:littleoh} hold. Because of the involution, which extends to the Dehn filled manifolds, the core curves of the Dehn fillings having equal length. Therefore, for large $n$ the cohomology stretch lamination $\Lambda(\rho_n)$ is exactly the union of the two Dehn filling core curves.
    \end{proof}
    \begin{figure}[h]
    \centering
    \includegraphics[scale=.35]{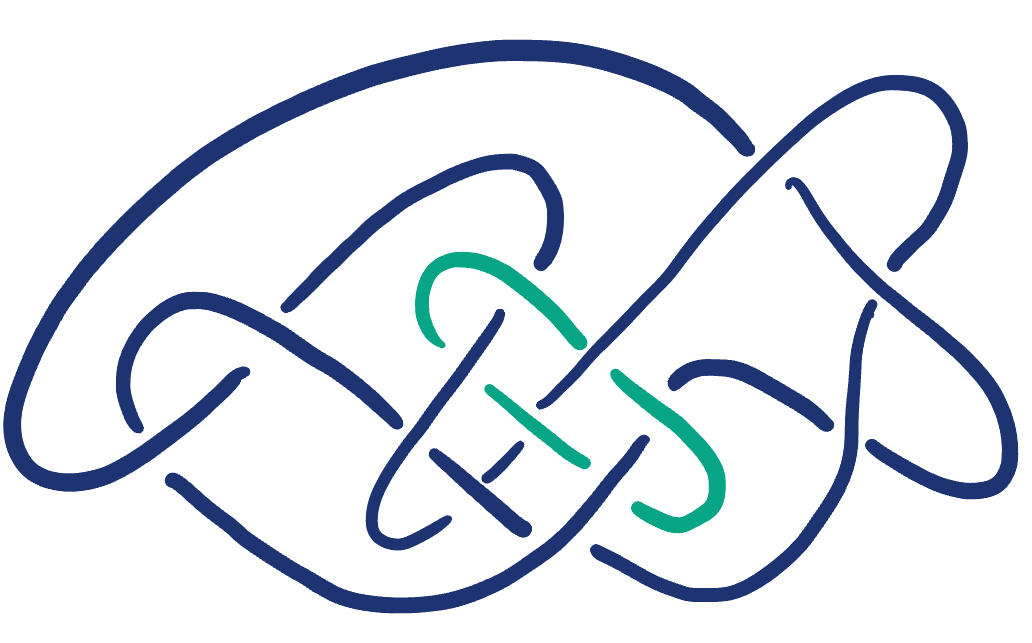}\caption{Link $L = L14n21792$, given in \cite{BD} as an example of a link whose exterior has the properties needed for Theorem \ref{thm:fiberedexamples}. See their paper for details on why this link has the desired properties.}\label{fig:link}
\end{figure}

    We end by noting that using \texttt{SnapPy} \cite{SnapPy}, one can compute the length of the core curves in these Dehn fillings and thus computationally explore the cohomology stretch of the classes in these examples.

\nocite{*}
\bibliographystyle{alpha}

\bibliography{bib}
\end{document}